\documentclass[reqno]{amsart}
\usepackage{amssymb}
\usepackage{amsfonts}
\usepackage{mathrsfs}
\usepackage[all]{xy}
\CompileMatrices
\input{diagxy}
\usepackage[active]{srcltx}

\renewcommand{\to}{\longrightarrow}
\newcommand{\dcup}{\mathinner{\cup \mkern -8.7mu \rlap{\raise 0.6ex\hbox{.}}\mkern 8.7mu}}
\newcommand{\ddbigcup}{\mathinner{\bigcup \mkern -15.3mu \rlap{\raise 0.7ex\hbox{.}}\mkern 14.9mu}}
\newcommand{\ov}{\overline}

\newcommand{\inv}{^{-1}}
\newcommand{\p}{\varphi}

\newcommand{\Thmname}{Theorem}
\newcommand{\Propname}{Proposition}
\newcommand{\Lemmaname}{Lemma}
\newcommand{\Definitionname}{Definition}
\newtheorem{Thm}{\Thmname}
\newtheorem{Prop}[Thm]{\Propname}
\newtheorem{Lemma}[Thm]{\Lemmaname}
{\theoremstyle{definition}
}
{\theoremstyle{remark}
}

{\theoremstyle{remark}
\newtheorem*{Claim}{Claim}}

\title{An elementary proof that subgroups of free groups are free}

\author{Benjamin Steinberg}
\address{School of Mathematics and Statistics \\ Carleton University \\
1125 Colonel By Drive\\
Ottawa, Ontario  K1S 5B6 \\
Canada}
\thanks{The author gratefully acknowledge the support of NSERC}
\email{bsteinbg@math.carleton.ca}
\date{\today}

\begin{document}
\begin{abstract}
We provide an elementary proof that subgroups of free groups are free via group actions.
\end{abstract}

\maketitle

\section{Introduction}
A group $F$ is \emph{free} on a set $X$ if there is a function $\iota\colon X\to F$ such that given any mapping $\rho\colon X\to G$ with $G$ a group, there is a unique homomorphism $\p\colon F\to G$ making the diagram
\begin{equation}\label{commutediag}
\xymatrix{X\ar[rr]^{\iota}\ar[rd]_{\rho}&& F\ar@{-->}[ld]^{\p}\\ & G &}
\end{equation}
commute. One can easily prove that $\iota$ is injective and the image of $\iota$ generates $F$.  One can then view elements of $F$ as words over the alphabet $X\cup X\inv$.  One proves that each word represents the same element of $F$ as a unique reduced word, called its reduced form;  a word is \emph{reduced} if it has no factor of the form $xx\inv$ with $x\in X\cup X\inv$. When convenient one identifies $F$ with the set of reduced words over $X\cup X\inv$. See~\cite{LyndonandSchupp} for details.

The classical Nielsen-Schreier theorem says that every subgroup of a free group is free.  The original proofs were combinatorial in nature, and therefore not very appealing.  Short conceptual proofs appeared later based on covering spaces of graphs and the Seifert-van Kampen theorem. However, this approach is beyond the scope of a first graduate algebra course.
The purpose of this note is to give an elementary, yet conceptual, proof that subgroups of free groups are free.
The idea is similar to an approach of the author and Ribes using wreath products~\cite{RibesSteinberg}, but we simplify things here by pursuing the avenue of group actions instead.  If $X$ is a nice topological space, then the category of covering spaces of $X$ is equivalent to the category of $\pi_1(X)$-sets, so it is clear that the topological proof should correspond to a group actions proof.  We assume nothing about free groups beyond what is in the previous paragraph.

\section{Group actions}
In this paper we shall work with right actions of groups.  The symmetric group on a set $A$ is written $S_A$.  The identity of a group is denoted $1$.

\subsection{Tensor products}
If $G$ is a group, then by a \emph{$G$-set}, we shall always mean a right $G$-set.  Let $H$ be a subgroup of $G$ and $A$ an $H$-set.  Then $H$ acts on the right of $A\times G$ by $(a,g)h = (ah,h\inv g)$.  The set of orbits of this action is denoted $A\otimes_H G$, as it is the natural notion of a tensor product in this context, cf.~\cite{MM-Sheaves}. The orbit of $(a,g)$ is denoted $a\otimes g$.  Observe that $A\otimes_H G$ is a right $G$-set via the action $(a\otimes g)g' = a\otimes gg'$.

Let $T$ be a transversal to the set of right cosets $G/H$ (i.e., a set of coset representatives) and denote by $\ov g$ the element of $T$ representing the coset $Hg$ for $g\in G$; we shall always assume that $1\in T$. Notice that the map $A\times G/H\to A\otimes_H G$ given by $(a,Hg)\mapsto a\otimes \ov g$ is a bijection.  Indeed, $a\otimes g = a\otimes g(\ov g)\inv \ov g=ag(\ov g)\inv\otimes \ov g$ and so the map $a\otimes g\mapsto (ag(\ov g)\inv,Hg)$ (which is easily checked to be well defined) is the inverse bijection.  We often identify $A\otimes_H G$ with $A\times G/H$ via this bijection.  The action of $G$ transfers via the bijection as:
\begin{equation}\label{actionincoords}
(a,Hg)g' = (a\ov gg'(\ov{gg'})\inv,Hgg').
\end{equation}
Notice that if $h\in H$, then $(a,H)h=(ah,H)$ using that $\ov 1=1=\ov h$.  Thus $A\times \{H\}$ is an $H$-invariant subset, isomorphic to $A$ as an $H$-set via the projection to the first coordinate.  From now on we do not distinguish these isomorphic $H$-sets.


\subsection{A characterization of free groups}
It is well known that two groups $G$ and $H$ are isomorphic if and only if the category of $G$-sets is equivalent to the category of $H$-sets.  Thus the following group action characterization of free groups should come as no surprise.

\begin{Prop}\label{charfree}
Let $X$ be a set and $F$ a group equipped with a map $\iota\colon X\to F$.  Then $F$ is a free group on $X$ (with respect to the mapping $\iota$) if and only if given any set $A$ and any map $\sigma\colon X\to S_A$, there is a unique action of $F$ on $A$ such that \[a\iota(x)=\sigma(x)(a)\] for all $x\in X$ and $a\in A$.
\end{Prop}
\begin{proof}
Clearly if $F$ is free on $X$, then it has the property described in the proposition.  For the converse, let $\rho\colon X\to G$ be a mapping with $G$ a group.  We need to construct a unique homomorphism $\p\colon F\to G$ such that the diagram \eqref{commutediag}
commutes.  Define $\sigma\colon X\to S_G$ by $\sigma(x)(g)=g\rho(x)$.   Then there is a unique action of $F$ on $G$ such that $g\cdot \iota(x)=g\rho(x)$ for all $g\in G$, $x\in X$ by hypothesis, where we use $\cdot$ to distinguish the action from multiplication in $G$. We claim that $h(g\cdot w) = (hg)\cdot w$ for all $h,g\in G$ and $w\in F$.  Indeed, fix $h\in G$.  It is immediately verified that the formula $g\odot w = h\inv[(hg)\cdot w]$ provides an action of $F$ on $G$.   Moreover, \[g\odot \iota(x) = h\inv[(hg)\cdot \iota(x)]=h\inv[hg\rho(x)] = g\rho(x).\]  Uniqueness now implies that $g\odot w=g\cdot w$ for all $w\in F$.  In other words, we have $h\inv[(hg)\cdot w] = g\cdot w$, or equivalently $(hg)\cdot w=h(g\cdot w)$,  for all $h,g\in G$ and $w\in F$.

Define $\p\colon F\to G$ by $\p(w)=1\cdot w$.  Then, for $x\in X$, one has $\p\iota(x)=1\cdot \iota(x)=1\rho(x)=\rho(x)$ by construction.  Furthermore, by the claim \[\p(v)\p(w) = \p(v)(1\cdot w) = (\p(v)1)\cdot w=(1\cdot v)\cdot w=1\cdot (vw)=\p(vw)\] and hence $\p$ is a homomorphism such that \eqref{commutediag} commutes.  It remains to verify that $\p$ is unique.  Suppose that $\tau\colon F\to G$ is another such homomorphism.  Define an action $\ast$ of $F$ on $G$ by $g\ast w= g\tau(w)$.  Then $g\ast \iota(x)=g\tau\iota(x)=g\rho(x)$ and so $\ast$ coincides with $\cdot$ by uniqueness.  Thus $\tau(w) = 1\ast w=1\cdot w=\p(w)$, as required.
\end{proof}

\section{The Nielsen-Schreier Theorem}
We now present an elementary proof that subgroups of free groups are free via group actions. Let $F$ be a free group on $X$ and $H$ a subgroup.

\subsection{Schreier Transversals}
A \emph{Schreier transversal} for $H\leq F$ is a transversal $T$ of $H$ in $F$ such that if we view $T$ as a set of reduced words, then $T$ is closed under taking prefixes (and hence in particular contains the empty word).  The existence of Schreier transversals is a straightforward application of Zorn's Lemma.  We include a proof for completeness.

\begin{Lemma}
There exists a Schreier transversal $T$ of $H$ in $F$.
\end{Lemma}
\begin{proof}
Consider the collection $\mathscr P$ of all prefix-closed sets of reduced words over $X\cup X^{-1}$ that intersect each right coset of $H$ in at most one element and order $\mathscr P$ by inclusion.  Then $\{1\}\in \mathscr P$, so it is non-empty.  It is also clear that the union of a chain of elements from $\mathscr P$ is again in $\mathscr P$, so $\mathscr P$ has a maximal element $T$ by Zorn's Lemma.  We need to show that each right coset of $H$ has a representative in $T$.
Suppose this is not the case and let $w$ be a minimum length word such that $Hw\cap T=\emptyset$.  Since $1\in T$, it follows $w\neq 1$ and hence $w=ux$ in reduced form where $x\in X\cup X\inv$.  By assumption on $w$, we have $Hu=Ht$ for some $t\in T$ (and so $Htx=Hw$).  If $tx$ is reduced as written, then $T\uplus \{tx\}\in \mathscr P$, contradicting the maximality of $T$.  If $tx$ is not reduced as written, then $tx$, or rather its reduced form, belongs to $T$ (by closure of $T$ under prefixes) and $Hw = Htx$. This contradicts the choice of $w$, completing the proof that $T$ is a transversal.
\end{proof}

\subsection{The Nielsen-Schreier Theorem} We now proceed to prove that subgroups of free groups using the tensor product construction.

\begin{Thm}[Nielsen-Schreier]
Subgroups of free groups are free.  More precisely, let $F$ be a free group on $X$ and let $H$ be a subgroup. Fix a Schreier transversal $T$ for $H$ and put
\begin{equation*}\label{Schreierbasis}
B=  \{tx(\ov {tx})\inv \mid (t,x)\in T\times X,\ tx(\ov{tx})\inv\neq 1\}.
\end{equation*}
 Then $H$ is freely generated by $B$.
\end{Thm}
\begin{proof}
By Proposition~\ref{charfree}, it suffices to show that given a map $\sigma\colon B\to S_A$, there is a unique action of $H$ on $A$ such that $ab=\sigma(b)(a)$.    For convenience, we extend $\sigma$ to $B\cup\{1\}$ by mapping $1$ to the identity of $S_A$.

First we prove uniqueness, as this will motivate the definition.  So assume we have such an action and consider the tensor product $A\otimes_H F$.  As usual we identify $A\otimes_H F$ with $A\times F/H$ where the action is given by $(a,Hv)w = (a\ov v w(\ov{vw})\inv,Hvw)$.  Our original action is the restriction of the action of $H$ on $A\times F/H$ to the subset $A\times \{H\}$ (under the usual identifications) and hence is uniquely determined by the action of $F$ on $A\otimes_H F$, which in turn is uniquely determined by the action of the generators $X$ of $F$.  But if $x\in X$, then
\begin{equation*}\label{defineaction}
(a,Hw)x = (a\ov wx(\ov {wx})\inv,Hwx)=(\sigma(\ov wx(\ov {wx})\inv)(a),Hwx)
\end{equation*}
(since $\ov wx(\ov {wx})\inv\in B\cup \{1\}$) and hence is uniquely determined by $\sigma$.

Let us now take $(a,Hw)x= (\sigma(\ov wx(\ov {wx})\inv)(a),Hwx)$ as the definition of an action of $F$ on $A\times F/H$; note that the action of $F$ in the second coordinate is the usual action of $F$ on $F/H$ and so $A\times \{H\}$ is invariant under $H$.    We must show that $(a,H)b=(\sigma(b)(a),H)$ for $b\in B$.

\begin{Claim}
Suppose that $t\in T$.  Then for all $a\in A$, one has $(a,H)t = (a,Ht)$.
\end{Claim}
\begin{proof}[Proof of claim]
We prove the claim by induction on the length of $t$ (as a reduced word).  If $t$ is empty, then trivially the claim holds.  Suppose first that $t=ux$ with $x\in X$ as a reduced word.  By definition of a Schreier transversal, we have $u\in T$ and so by induction $(a,H)ux=(a,Hu)x=(\sigma(ux(\ov {ux})\inv)(a),Hux)$.  But $ux=t=\ov{ux}$ and so the right hand side is $(a,Hux)$ as required.

Next suppose that $t=ux\inv$ with $x\in X$ (as a reduced word).   By definition of a Schreier transversal, $u\in T$ and so $\ov{tx}=u=tx$.  We need to verify that $(a,Ht)=(a,H)t$, or equivalently, that $(a,Ht)x=(a,H)u$.  But $(a,H)u=(a,Hu)$ by induction.  On the other hand, $(a,Ht)x= (\sigma(tx(\ov{tx})\inv)(a),Htx)=(a,Hu)$ establishing the claim.
\end{proof}

To complete the proof, we must show that if $t\in T$ and $x\in X$ with $tx(\ov{tx})\inv\neq 1$, then $(a,H)tx(\ov{tx})\inv = (\sigma(tx(\ov{tx})\inv)(a),H)$; or equivalently, we must show that
\begin{equation}\label{thecomputation}
(a,H)tx=(\sigma(tx(\ov{tx})\inv)(a),H)\ov{tx}.
\end{equation}
By the claim, the right hand side of \eqref{thecomputation} is $(\sigma(tx(\ov{tx})\inv)(a),Htx)$, whereas the left hand side is $(a,Ht)x = (\sigma(tx(\ov{tx})\inv)(a),Htx)$. This completes the proof that $H$ is free on the set $B$.
\end{proof}

It is an easy combinatorial exercise to verify that the elements of $B$ are distinct (the above proof does not provide this) and to count that if the size of $X$ is $n$ and $[F:H]=m$, then $B$ has $1+m(n-1)$ elements (this is Schreier's formula).  For the last statement one just observes that $T\times X$ has $mn$ elements and that, for each non-empty word $t\in T$, there is exactly one element $x\in X$ so that $tx(\ov{tx})\inv =1$.

\end{document}